\documentclass[11pt]{article}

\usepackage[margin=1in]{geometry}
\usepackage{graphicx}
\usepackage{multirow}
\usepackage{amsmath,amssymb,amsfonts,amsthm}
\usepackage{mathrsfs}
\usepackage{xcolor}
\usepackage{textcomp}
\usepackage{manyfoot}
\usepackage{booktabs}
\usepackage{algorithm}
\usepackage{algorithmicx}
\usepackage{algpseudocode}
\usepackage{listings}
\usepackage{enumitem}
\usepackage{subcaption}
\usepackage[capitalise]{cleveref}

\usepackage{tikz}
\usepackage{pgfplots}
\pgfplotsset{compat=1.14}

\usepackage[sort&compress,numbers]{natbib}
\newcommand{\tr}{^\intercal}
\newcommand{\R}{\mathbb{R}}
\newcommand{\Z}{\mathbb{Z}}
\newcommand{\N}{\mathbb{N}}
\newcommand{\D}{\mathcal{X}}
\newcommand{\cH}{\mathcal{H}}
\newcommand{\cL}{\mathcal{L}}
\newcommand{\B}{\mathcal{B}}
\newcommand{\SP}{\operatorname{SP}}
\newcommand{\st}{\operatorname{s.t.}}
\newcommand{\dsum}{\displaystyle\sum}
\newcommand{\nint}{\operatorname{nint}}
\newcommand{\dist}{\operatorname{dist}}
\newcommand{\argmin}{\operatorname{argmin}}

\theoremstyle{plain}
\newtheorem{theorem}{Theorem}
\newtheorem{proposition}{Proposition}
\newtheorem{lemma}{Lemma}
\newtheorem{corollary}{Corollary}

\theoremstyle{definition}
\newtheorem{example}{Example}

\theoremstyle{remark}
\newtheorem{remark}{Remark}

\title{Lagrangian Reformulation for Nonconvex Optimization:\\
Tailoring Problems to Specialized Solvers}

\author{
Rodolfo A. Quintero\thanks{Industrial and Systems Engineering, Lehigh University. Email: \texttt{roq219@lehigh.edu}}
\and
Juan C. Vera\thanks{Econometrics and Operations Research, Tilburg University. Email: \texttt{j.c.veralizcano@tilburguniversity.edu}}
\and
Luis F. Zuluaga\thanks{Industrial and Systems Engineering, Lehigh University. Email: \texttt{luis.zuluaga@lehigh.edu}}
}

\date{} 

\begin{document}
\maketitle

\begin{abstract}
In recent years, the interest in reformulating nonconvex optimization problems, particularly those involving binary variables, has surged significantly. This growing attention stems from remarkable advancements in computing technologies, such as quantum and Ising devices, along with notable improvements in both quantum and classical optimization solvers, which leverage specific formulations to solve nonconvex problems.

Our research characterizes the equivalence between equality-constrained nonconvex optimization problems and their Lagrangian relaxation, thereby enabling the application of the aforementioned cutting-edge technologies to solve optimization problems effectively.
In addition to addressing a significant gap in the literature, our results are readily applicable to many important situations in practice.
To achieve these results, we link specific characteristics of optimization problems to broader, classical insights regarding Lagrangian duality in the context of general nonconvex problems. Moreover, our comprehensive approach to examining the equivalence between problem formulations considers not only the objective value but also the optimal solution. This perspective, often overlooked in existing literature, is particularly relevant for problems that involve both continuous and binary variables.
\end{abstract}

\section{Introduction}
\label{sec:intro}
There is considerable freedom in the way that a given optimization problem can be formulated to leverage the particular algorithm or computing hardware that will be used to solve the problem~\citep[][]{liberti2010reformulation}. This flexibility becomes particularly crucial for nonconvex optimization problems given how complex they are to solve~\citep[][]{cartis2022evaluation}.
The landscape of computational devices has evolved dramatically, introducing novel paradigms, such as quantum~\citep[][]{preskill2018quantum, rajak2023quantum} and {\em Ising}~\citep[][]{mohseni2022ising} devices. Concurrently, algorithmic advances, such as {\em quantum annealing}~\citep{santoro2006optimization}, the {\em quantum approximation optimization algorithm} (QAOA)~\citep{farhi2014quantum, hadfield2019quantum}, the {\em quantum hamiltonian descent} (QHD)~\citep{2023QHD}, and a variety of classical global optimization solution methodologies and solvers~\citep{kronqvist2019review, letourneau2023efficient}, offer diverse avenues for addressing particular formulations of nonconvex optimization problems. This proliferation has reignited interest in reformulation techniques tailored to this type of problems~\citep[][]{lasserre2016max, lucas2014ising, gusmeroli2022expedis, quintero2022characterization, date2021qubo, bartmeyer2024benefits}.

Consider, for instance, the {\em quadratic unconstrained binary optimization} (QUBO) reformulations of combinatorial optimization problems~\cite  [][]{punnen2022quadratic}.
A QUBO problem, as its name indicates, is of the general form $\min\{x\tr Q x + c\tr x: x \in \{0,1\}\}$. QUBO encompasses the max-cut problem and the closely related {\em Ising model}~\citep[][]{cipra1987introduction, gaudioso2020view}. Given that Ising devices~\citep[][]{mohseni2022ising}, as well as {\em annealing}~\citep[][]{kumar2018quantum} and {\em gate-based}~\citep[][]{zahedinejad2017combinatorial} quantum devices can address the solution of QUBO problems, and potentially have {\em quantum supremacy}~\citep[][]{preskill2012quantum} over classical computers on this task, it is clearly of interest to study whether or how combinatorial optimization problems that do not have a natural QUBO formulation (e.g., the stable set problem) can be reformulated as a QUBO~\citep[][]{lasserre2016max, lucas2014ising, quintero2022characterization, nannicini2019performance, gusmeroli2022expedis}.
Several combinatorial optimization problems have been reformulated as QUBO's in a rather ad-hoc way, that is, in a case by case way. Similar, but more general, than~\cite{lasserre2016max, gusmeroli2022expedis}, our work proposes a tool to reformulate in an standard way a general class of combinatorial optimization problems.

A classical approach to solving optimization problems with complex or computationally intensive constraints is the use of Lagrangian relaxations~\citep[][]{guignard2003lagrangean, gaudioso2020view}. By relaxing (or dualizing) the constraints and solving the resulting simpler problem iteratively, Lagrangian relaxation provides a systematic approach to finding near-optimal solutions efficiently. The beauty of this approach lies in the associated Lagrangian-duality theory, which offers profound insights into the quality of the obtained solutions and guides the optimization process. However, it is important to note that Lagrangian relaxation often leads to suboptimal solutions; generically, there are no convergence guarantees, and iteratively solving the relaxed problem could be computationally intensive, especially in the nonconvex case.

Our research characterizes the equivalence between equality-constrained nonconvex optimization problems and their Lagrangian relaxation, thereby enabling the aforementioned new technologies to solve these problems.
Specifically,
given continuous functions $f, h_1,\dots,h_m$ defined in a compact set $\D \subseteq \R^n$, consider the following optimization problem
\begin{equation}
\tag{P}
\label{mod:P_D}
\begin{array}{lllll}
\displaystyle \min& f(x) \\
\st & h_j(x) = 0, & j=1,\dots,m,\\
     & x \in \D.
\end{array}
\end{equation}
Given $y \in \R^m$ define, the Lagrangian relaxation
\begin{equation}\label{mod:LDy}\tag{$D_{y}$}
\min_{x\in\D}\left(f(x)+\sum_{j=1}^m y_jh_j(x)\right).
\end{equation}
We are interested in characterizing when $\eqref{mod:P_D}$ has a {\em Lagrangian reformulation}. That is, the cases when for some $y \in \R^m$, the optimal values of $\eqref{mod:P_D}$ and $\eqref{mod:LDy}$ coincide. If furthermore, the set of optimal solutions of $\eqref{mod:P_D}$ and $\eqref{mod:LDy}$ coincide, we say that $\eqref{mod:P_D}$ and $\eqref{mod:LDy}$ are {\em equivalent}.

Notice that having a Lagrangian reformulation is equivalent to the strong duality between~\eqref{mod:P_D}
and its (Lagrangian) dual, plus dual attainment. We prove general versions of these results in Section~\ref{sec:reform}. Namely, in Proposition~\ref{prop:dualstrong}, we prove the needed strong duality result, and in Theorem~\ref{thm:reform}, by characterizing dual attainment, we prove the desired Lagrangian reformulation. Beyond
analyzing the equivalence between problems~\eqref{mod:P_D} and~\eqref{mod:LDy} in terms of their optimal objective,
in Theorem~\ref{thm:sol_convergence}, we further
analyze the equivalence between problems~\eqref{mod:P_D} and~\eqref{mod:LDy} in terms of their optimal solutions.

Naturally, the results described above hold under appropriate assumptions.
To establish the equivalence, we start with the assumption that \( h_j(x) \ge 0 \) on $\D$. Since we are focusing on the non-linear case, this assumption is not limiting. We can substitute each \( h_j(x) \) with \( h_j(x)^2 \), \( |h_j(x)| \), or more generally, with \( \phi(h_j(x)) \), where \( \phi: \mathbb{R} \to \mathbb{R}_+ \) is a function that satisfies \( \phi(x) = 0 \) if and only if \( x = 0 \).

In Proposition~\ref{prop:dualattain}, we characterize the conditions under which problem~\eqref{mod:P_D} has a Lagrangian reformulation. This characterization is based on an abstract equivalent condition related to the behavior of the functions \(f\) and \(h\) over the set \(\D \setminus \cH\). In Theorem~\ref{thm:reform}, we present condition, which is both sufficient and easier to verify. Specifically, let $\cH = \{x \in \R^n: h_j(x) = 0, j=1,\dots,m\}$, then
  problem~\eqref{mod:P_D} has a Lagrangian reformulation if the set \(\D \setminus \cH\) is closed. Since \(\D\) is compact and \(\cH\) is closed, this condition is equivalent to the set \(\D\) being the disjoint union of the two compact sets \(\D \cap \cH\) and \(\D \setminus \cH\). While this condition may appear overly strict, it is essential for applying the Lagrangian reformulation property in practical cases of nonconvex optimization problems (see Section~\ref{sec:apps}), especially when the set \(\D\) is discrete. However, as shown in Examples~\ref{ex:yes} and~\ref{ex:no},  it is not necessary for the set \(\D \setminus \cH\) to be closed to achieve the desired Lagrangian reformulation property.

By leveraging the general Lagrangian reformulation results described above, in Section~\ref{sec:apps}, we show how distinctive nonconvex optimization problems can be reformulated in a way that is amenable to new computing and algorithmic technologies.

In Section~\ref{sec:QUBOapps}, we apply our results to obtaining QUBO reformulations of combinatorial optimization problems. As mentioned earlier, this type of reformulation is relevant because it allows solving combinatorial optimization problems using Ising devices, as well as annealing and gate-based quantum computers,  using, respectively, annealing, quantum annealing, and QAOA algorithms.
In this direction, our general Lagrangian reformulation characterization allows us to generalize and improve related results in the literature. In particular, consider~\citet{lasserre2016max}, which obtains QUBO reformulations of pure binary, linearly equality-constrained quadratic optimization problems. This approach is extended to problems including linear inequality constraints by adding so-called slack variables.
Our results are not limited to problems with linear constraints. 
Thus, we can use them to obtain QUBO reformulations of pure binary, quadratically- and linearly-constrained quadratic optimization problems. We illustrate the relevance of this generalization by obtaining alternative QUBO reformulations of general vertex cover problems. Namely, starting by equivalently formulating the vertex cover problem with quadratic equality constraints (instead of linear inequality constraints), we obtain a QUBO reformulation without adding auxiliary slack variables. In practice, this is consequential, as the number of {\em qubits} (i.e., quantum bits) required to {\em embed} (i.e., encode)~\citep[][]{zbinden2020embedding} a QUBO problem on a quantum device can grow quadratically~\citep{date2019efficiently} with its number of variables. This growth is due to limitations on qubit connectivity on current and near-term quantum devices~\citep[][]{king2020performance}.

In Section~\ref{sec:MIQPapps}, we consider the more general setting of mixed-integer optimization, where the goal is to minimize a function over a feasible set defined by linear constraints with continuous and binary variables with known upper and lower bounds. This problem class encompasses a wide range of practical optimization problems. We characterize how different Lagrangian reformulations can be obtained by relaxing (or dualizing) ($i$) linear constraints, ($ii$) binary constraints, or ($iii$) both (see Theorems~\ref{thm:equivBPpure} and~\ref{thm:equivBPmix}). In the first case, the reformulation is suitable for Ising machines and quantum devices using quantum annealing or QAOA algorithms when all variables are binary; effective methods for box-constrained mixed-integer or continuous polynomial optimization when the objective is a polynomial~\citep{buchheim2014box, letourneau2023efficient}; and branch-and-bound techniques combined with methods for box-constrained quadratic optimization when the objective is quadratic~\citep[][]{kim2010tackling, bonami2018globally}. In the second and third cases, the reformulations are amenable to a wide range of mixed-integer nonlinear optimization solvers~\citep[][for a recent benchmark of such solvers]{kronqvist2019review}, without requiring branch-and-bound techniques for binary variables. Additionally, in the third case, the reformulation is amenable to recent quantum optimization algorithms like {\em quantum gradient descent}~\citep{rebentrost2019quantum}, quantum Hamiltonian descent (QHD)~\citep{2023QHD}, and {\em quantum Langevin dynamics}~\citep{chen2023quantum}, designed for quantum devices. Moreover, in Theorems~\ref{thm:equivBPpure} and~\ref{thm:equivBPmix}, we examine Lagrangian reformulation equivalence not only in terms of the objective but also from the viewpoint of the problem's solutions. This perspective, often overlooked in the literature, is crucial for these second and third cases. We characterize when rounding (to the nearest integer) the solution of the Lagrangian reformulation provides the values of the integer variables of a mixed-integer solution to the original problem (see Theorem~\ref{thm:equivBPmix}\ref{it:2a}-\ref{it:2b}).

Our results bridge a gap between two bodies of work. On the one hand, general Lagrangian duality results for nonconvex optimization~\citep[][among many others]{rockafellar1974augmented, bertsekas1976penalty, huang2003unified, dolgopolik2016unifying, estrin2020implementing} often lack constructive dual attainment characterizations necessary for deriving the practical reformulations that are of interest here. On the other hand, Lagrangian reformulation results exist for specific optimization problem classes, such as linearly constrained pure binary quadratic optimization~\citep[][]{lasserre2016max, lucas2014ising, bartmeyer2024benefits, date2021qubo, quintero2022characterization, gusmeroli2022expedis}, mixed-integer linear optimization~\citep{feizollahi2017exact}, and mixed-integer convex quadratic optimization~\citep{bhardwaj2022exact, gu2020exact}. By deriving constructive strong duality and dual attainment results, we provide Lagrangian reformulation results with practical applications for a broad class of nonconvex optimization problems, including the latter problems mentioned. Furthermore, our results extend theoretical Lagrangian reformulation findings for linear or convex quadratic objectives~\citep{feizollahi2017exact, bhardwaj2022exact, gu2020exact} to the more challenging nonconvex objective case. Moreover, while these works analyze Lagrangian reformulations obtained from a relaxation of the linear constraints in the original problem; here, we consider a much more general class of Lagrangian reformulations obtained, as described above, from relaxing both linear and nonlinear constraints (and possibly nonconvex, such as binary constraints)  in the original problem.

The remainder of the article is organized as follows. Section~\ref{sec:prelim} introduces definitions and basic results used throughout the article. Section~\ref{sec:reform} presents the main results that characterize cases where problem~\eqref{mod:P_D} has a Lagrangian relaxation reformulation. Section~\ref{sec:apps} applies these general results of Section~\ref{sec:reform}
to obtain QUBO reformulations of combinatorial optimization problems (Section~\ref{sec:QUBOapps}) and different Lagrangian reformulations of linearly constrained mixed-integer optimization problems amenable to novel computing devices and solution algorithms
(Section~\ref{sec:MIQPapps}).

\section{Preliminaries}
\label{sec:prelim}
For $x\in \D$ and $y\in \R^m$,
consider the Lagrangian function
\[\cL(x,y) := f(x) + \sum_{j=1}^m y_jh_j(x)\]
 for~\eqref{mod:P_D}. Then, the Lagrangian dual of~\eqref{mod:P_D} is given by:
\begin{equation}
\tag{$D$}
\label{mod:POL}
\displaystyle \sup_{y\in \R^m} \inf_{x \in \D}  f(x)+\dsum_{j=1}^m  y_jh_j(x).
\end{equation}
A basic property of the Lagrangian dual is weak duality.
\medskip
\begin{proposition}[Weak Duality~(Thm. 12.11,~\citep{nocedal2006numerical})]
\label{pro:weak}
$\eqref{mod:P_D}^*  \ge \eqref{mod:POL}^*$.
\end{proposition}
As in the standard practice, we use $(\cdot)^*$ above to denote the optimal value of the optimization problem $(\cdot)$.
In particular, note that stating that~\eqref{mod:P_D} has a Lagrangian reformulation is equivalent to stating there is $y \in \R^m$ such that
\begin{equation*}
\eqref{mod:P_D}^* = \eqref{mod:LDy}^*.
\end{equation*}
Further, we define the optimal value of an infeasible (resp. unbounded) minimization problem by $+\infty$ (resp. $-\infty$).

Now, let
\[h(x):= \sum_{j=1}^m h_j(x),\]
and let
\begin{equation}
\label{eq:Hdef}
\cH = \{x \in \R^n: h_j(x) = 0, j=1,\dots,m\}.
\end{equation}
and  $d(x, \cH): = \min\{\|x-z\|: z \in \cH\}$ be the  distance from $x \in \R^n$ to $\cH$. Also,
for any $\epsilon > 0$, let
\begin{equation}
\label{eq:heps}
h^{\epsilon} :=
\min\{h(x): x \in \D, \ d(x, \cH) \ge \epsilon\},
\end{equation}
\begin{equation}
\label{eq:feps}
f_{\epsilon} := \min \{ f(x): x\in \D, \ d(x, \cH)\leq \epsilon\},
\end{equation}
and
\begin{equation}
\label{eq:hepslim}
\tilde h = \lim_{\epsilon \downarrow 0} h^{\epsilon}.
\end{equation}
The following lemma states characteristics of these quantities that will be useful therein.

\begin{lemma}
\label{lem:heps}
Let $f$, $h_j$ for $j=1, \ldots, m$ be real valued continuous functions and $\D \subseteq \R^n$ be a compact set. Assume $h_j(x) \ge 0$ for all $x \in \D$, $j=1,\dots,m$ and $\D\cap \mathcal{H}\neq \emptyset$ and $\D\setminus \cH \neq \emptyset$. Then

\begin{enumerate}[label = (\roman*)]
\item $h^\epsilon$ is non-decreasing in $\epsilon$, for any $\epsilon > 0$; $h^{\epsilon} >0$ for any $\epsilon > 0$; and $\tilde h \ge 0$.
\label{item:hepsone}
\item $f_\epsilon$ is non-increasing in $\epsilon$, for any $\epsilon > 0$ and $\lim_{\epsilon \downarrow 0} f_\epsilon = \eqref{mod:P_D}^*$.
\label{item:hepstwo}
\end{enumerate}
\end{lemma}

\begin{proof} Statement \ref{item:hepsone}: The non-decreasing characteristic follows by definition (see~\eqref{eq:heps}). Now take $\epsilon >0$.
If $S:= \{x \in \D: \ d(x, \cH) \ge \epsilon\}$ is empty, $h^{\epsilon} = + \infty$. Otherwise, there is $x \in S$ such that $h^{\epsilon} = h(x) > 0$. This implies both that $h^{\epsilon} >0$ for any $\epsilon > 0$, and $\tilde h \ge 0$.
Statement \ref{item:hepstwo}: From~\eqref{eq:feps}, it follows that $f_{\epsilon}$ is non-increasing in $\epsilon$  and bounded above by $f_0$. Therefore, $\lim_{\epsilon\downarrow 0}f_{\epsilon}$ exists and $\lim_{\epsilon\downarrow 0}f_{\epsilon} \le f_0 =  \eqref{mod:P_D}^*$.
Next, we prove that $\lim_{\epsilon\downarrow 0}f_{\epsilon} \ge \eqref{mod:P_D}^*$, which concludes the proof.
Let $(x_\ell)_{\ell \in \N}\subset \D$ be a sequence generated by taking $x_\ell\in \argmin\{f(x): x\in \D, \ d(x, \cH)\leq 1/\ell\} $.
Because $\D$ is compact, $(x_\ell)_{\ell\in \N}$ has a convergent subsequence $(x_{\ell_k})_{k \in \N}$. Let $\hat{x} = \lim_{k\to \infty} x_{\ell_k}$.
We have, $\hat{x} \in \D$ because~$\D$ is compact. Also,
\begin{align*}
d(\hat{x},\cH) = d\left(\lim_{k\to \infty} x_{\ell_k}, \cH\right) = \lim_{k\to \infty} d(x_{\ell_k},\cH) \leq \lim_{k\to \infty} \tfrac{1}{\ell_k} = 0,
\end{align*}
and thus $\hat{x} \in \cH$ because $\cH$ is closed. Thus $\hat x \in \D\cap \cH$. Therefore
\[\lim_{\epsilon\downarrow 0}f_{\epsilon} = \lim_{k \to \infty }f(x_{\ell_k}) = f(\hat{x}) \ge \eqref{mod:P_D}^*.\]
\end{proof}

The following lemma characterizes conditions under which the optimal value of the Lagrangian relaxation~\eqref{mod:LDy} of~\eqref{mod:P_D}, loosely speaking, gives an epsilon approximation for~\eqref{mod:P_D}$^*$. This result is crucial in proving our general strong duality (Proposition~\ref{prop:dualstrong}) and Lagrangian reformulation (Theorem~\ref{thm:reform}) results.

\begin{lemma}\label{lem:elcaso}
Let $f$, $h_j$ for $j=1, \ldots, m$ be real valued continuous functions and $\D \subseteq \R^n$ be a compact set. Assume $h_j(x) \ge 0$ for all $x \in \D$, $j=1,\dots,m$, $\D\cap \mathcal{H}\neq \emptyset$, and $\D\setminus \mathcal{H}\neq \emptyset$.
Let $\epsilon > 0$ be such that $h^\epsilon < \infty$ and $y \in \R^m$ be such that $y_j \ge \frac{f_{\epsilon} - (D_0)^*}{h^\epsilon}$, $j=1,\dots,m$.  Then, $\inf_{x \in \D} f(x) + \sum_{j=1}^m y_j h_j(x) \ge f_{\epsilon}$.
\end{lemma}

\begin{proof}
Assume $\D\cap \mathcal{H}\neq \emptyset$ and $\D\setminus \cH \neq \emptyset$.
Let $\epsilon>0$ be such that $h^\epsilon < \infty$, then
$(D_0)^* \le f_{\epsilon} $, and $h^\epsilon > 0$ by Lemma~\ref{lem:heps}\ref{item:hepsone}. Let
\begin{equation}\label{eq:fyj}
y_j \ge \frac{f_{\epsilon} - (D_0)^*}{h^\epsilon} \ge 0 \text{ for } j=1,\dots,m.
 \end{equation}
 Now, let $x \in \D$. If $d(x,\cH) \le \epsilon$, we have
 \begin{equation}\label{eq:<eps}
f(x) + \sum_{i=1}^m
y_j h_j(x) \ge f(x) \ge f_{\epsilon}.
 \end{equation}
Otherwise, if $d(x,\cH) \ge \epsilon$,  we have $h^{\epsilon} \le h(x)$ and thus,
 \begin{equation}\label{eq:>eps}
 f(x) + \sum_{i=1}^m
 y_j h_j(x) \ge f(x) + f_{\epsilon} -(D_0)^* \ge f_{\epsilon}.
 \end{equation}
Equations~\eqref{eq:<eps} and~\eqref{eq:>eps}
imply
 \[
   \inf_{x \in \D} f(x) + \sum_{j=1}^m
   y_j h_j(x) \ge f_{\epsilon}.
   \]
\end{proof}

As it will be made clear in Section~\ref{sec:reform}, the behaviour of $\tilde{h}$ is tightly related to dual attainment and hence the Lagrangian reformulation property. Lemma~\ref{lem:htilde>0} below relates the behaviour of $\tilde{h}$ with a characteristic of the feasible set of~\eqref{mod:P_D}. This relationship is crucial to obtain practically relevant Lagrangian reformulation results in Section~\ref{sec:apps}.

\begin{lemma}\label{lem:htilde>0}
Let $h_j$ for $j=1, \ldots, m$ be real-valued continuous functions and $\D \subseteq \R^n$ be a compact set. Assume $h_j(x) \ge 0$ for all $x \in \D$, $j=1,\dots,m$. Then $\tilde h >0$ if and only if $\D\setminus \cH$ is a closed set.
\end{lemma}

\begin{proof}
If $\D \setminus \cH$ is empty then $\D \setminus \cH$ is closed and $\tilde{h} = +\infty$. Thus, assume that $\D \setminus \cH \neq \emptyset$. Assume $\tilde h >0$. Let $(x_\ell)_{\ell\in \N} \subseteq \D\setminus \cH$ be a convergent sequence with limit point $\hat x$.
As $\D$ is compact $\hat x \in \D$.
By definition of $\hat x$ and the continuity of $h$, there exists $\ell \in \N$ such that $|h(\hat x) - h(x_\ell)|<\tilde h$.
Let $\hat \epsilon = d(x_\ell, \cH)$. Then, by Lemma~\ref{lem:heps}\ref{item:hepsone}, $\tilde h \le h^{\hat \epsilon} \le h(x_\ell)$. We have then
$|h(\hat x)| \ge  |h(x_\ell)| - |h(x_\ell) - h(\hat x)| > 0$. Therefore, $\hat x \notin \cH$.

Conversely, assume $\D\setminus \cH$ is closed. Then, $\D\setminus \cH$ is compact since $\D$ is compact. Let $(x_\ell)_{\ell\in \N} \subseteq \D\setminus \cH$ be a sequence satisfying $x_\ell\in \argmin \{h(x): x\in \D, \ d(x, \cH)\geq 1/\ell\}$. As $\D\setminus\cH $ is compact, there is a subsequence $(x_{\ell_k})_{k\in \N}$ converging to some $\hat x \in \D\setminus \cH$. Then $\tilde h = \lim_{k\to \infty}h(x_{\ell_k}) = h(\hat x)>0$.
\end{proof}

As mentioned earlier, our results apply to instances of~\eqref{mod:P_D} defined by  continuous functions $f, h_1,\dots,h_m$ defined in a compact set $\D \subseteq \R^n$. An exception is Theorem~\ref{thm:equivBPpure}, in Section~\ref{sec:apps}, were we consider functions that satisfy the additional mild condition of being $L$-smooth on $\D$. Specifically,  $f: \R^n \to \R$ is $L$-smooth if it is continuously differentiable and its gradient is Lipschitz continuous with Lipschitz constant $L$ on $\D$; that is, $\|\nabla f(x) - \nabla f(y)\| \le L \|x-y\|$ for all $x, y \in \D$~\citep[][Def.~1]{jalilzadeh2022variable}.

\section{Characterizing the Lagrangian reformulation}
\label{sec:reform}
The characterization of the Lagrangian reformulation is conducted in two steps. First, we present a general result regarding the strong duality between problems~\eqref{mod:P_D} and~\eqref{mod:POL}. Next, we characterize the conditions under which the dual optimal value is attained, which is equivalent to establishing a Lagrangian reformulation. Our approach is constructive, enabling us to derive important results about whether and how Lagrangian reformulations of specific practical problems can be achieved.

\subsection{Strong Duality} Numerous proofs of strong duality for nonlinear optimization are available in both classical and contemporary literature (see, e.g., \citep{rockafellar1973dual, mangasarian1975unconstrained}). However, rather than focusing on a general nonlinear optimization framework, we specifically target equality-constrained problems, which allows us to derive simpler proofs.

\begin{proposition}[Strong duality]
\label{prop:dualstrong}
Let $f$, $h_j$ for $j=1, \ldots, m$ be real valued continuous functions and $\D \subseteq \R^n$ be a compact set. Assume $h_j(x) \ge 0$ for all $x \in \D$, $j=1,\dots,m$. Then, \eqref{mod:P_D}$^* =$ \eqref{mod:POL}$^*$.
\end{proposition}

\begin{proof}
From weak duality (Proposition~\ref{pro:weak}), it is
enough to show that $\eqref{mod:P_D}^* \le \eqref{mod:POL}^*$.  We analyze two cases: First, assume
$\D\cap \cH=\emptyset$, then, \eqref{mod:P_D}$^*=+\infty$. We
show that \eqref{mod:POL}$^* = +\infty$. Since $\D$ is compact, and $h(x) :=\sum_{j=1}^m h_j(x)$ is
continuous, there is $\epsilon>0$ such that $h(x)\geq\epsilon$ for every $x\in \D$. For any $k\geq (D_0)^*$ we have
\begin{align*}
   \sup_{y\in \R^m}\inf_{x\in \D}f(x)+\sum_{j=1}^m  y_jh_j(x) \ge  \inf_{x\in \D} f(x) + \left( \frac{k - (D_0)^*}{\epsilon}\right) h(x) \ge  k,
\end{align*}
and therefore, \eqref{mod:POL}$^* = +\infty$.

Now, assume $\D\cap \mathcal{H}\neq \emptyset$. If $\D\setminus \cH = \emptyset$, then, $\D\cap \cH=\D$, and it follows directly that $\eqref{mod:P_D}^*=\eqref{mod:POL}^*$. Thus, also assume $\D\setminus \cH \neq \emptyset$. Then from Lemma~\ref{lem:elcaso}, it follows that for every $\epsilon>0$ small enough
 \begin{align}\label{eq:up_bd_pepsilon}
 \eqref{mod:POL}^* = \sup_{y \in \R^m} \inf_{x \in \D} f(x) + \sum_{j=1}^m y_j h_j(x) \ge f_{\epsilon}.
 \end{align}
Now from Lemma~\ref{lem:heps}\ref{item:hepstwo}, taking the limit when $\epsilon \downarrow 0$ in~\eqref{eq:up_bd_pepsilon}, we obtain
$
 \eqref{mod:POL}^* \ge \lim_{\epsilon \downarrow 0} f_{\epsilon} = \eqref{mod:P_D}^*.
 $
\end{proof}

The following lemma shows that once equipped with the strong duality (Proposition~\ref{prop:dualstrong}), one obtains a convergent sequence to~$\eqref{mod:P_D}^*$ by solving~$\eqref{mod:LDy}^*$ with increasing values of the Lagrangian multipliers $y_j$, $j=1,\dots,m$.

\begin{lemma}\label{lem:lim=sup}
Let $f$, $h_j$ for $j=1, \ldots, m$ be real valued continuous functions and $\D \subseteq \R^n$ be a compact set. Assume $h_j(x) \ge 0$ for all $x \in \D$, $j=1,\dots,m$. Then
\[\lim_{y_1,\dots,y_m \to \infty} \eqref{mod:LDy}^* = \eqref{mod:P_D}^*.\]
\end{lemma}
\begin{proof} Since $h_j(x) \ge 0$ for all $x\in \D$, the function  $\eqref{mod:LDy}^*$ is non-decreasing in $y_j$, $j = 1,\dots,m$. Therefore $\lim_{y_1,\dots,y_m \to \infty} \eqref{mod:LDy}^* = \sup_{y \in \R^m} \eqref{mod:LDy}^* = \eqref{mod:POL}^*$. The result follows then from strong duality.
\end{proof}

\subsection{Lagrangian Reformulation}

The result in Lemma 4 is intuitive, as higher values of \( y \) correspond to lower constraint violation. Now, we present a stronger result than Lemma~\ref{lem:lim=sup}. In Proposition~\ref{prop:dualattain}, we characterize the conditions under which problem~\eqref{mod:P_D} has a Lagrangian reformulation. This characterization is based on the behavior of \( f \) and \( h \) on the set \( \D \setminus \cH \). Specifically, we determine when \( \eqref{mod:P_D}^* = \eqref{mod:LDy}^* \) for a finite value of the Lagrangian multipliers \( y_j \) for \( j=1, \ldots, m \).

\begin{proposition}
\label{prop:dualattain}
Let $f$, $h_j$ for $j=1, \ldots, m$ be real valued continuous functions and $\D \subseteq \R^n$ be a compact set. Assume $h_j(x) \ge 0$ for all $x \in \D$, $j=1,\dots,m$. Then, there is $y \in \R^m$ such that~\eqref{mod:LDy} is a Lagrangian reformulation of~\eqref{mod:P_D}
 if and only if $\tfrac{\eqref{mod:P_D}^*-f(x)}{\sum_{j=1}^m h_j(x)}$ is bounded above on $\D \setminus \cH$.
\end{proposition}

\begin{proof}
Let $h(x) = \sum_{j=1}^m h_j(x)$.
We prove the contrapositive statement of the if and only if statement. Assume there is no $y\in \R^m$ such that
$\eqref{mod:P_D}^* = \eqref{mod:LDy}^*$ holds.
This  implies that for all  $\ell \in \N$ there is $x_\ell \in \D$ such that
\begin{equation}
\label{eq:prelim}
\eqref{mod:P_D}^*> f(x_\ell)+\ell h(x_\ell) \ge f(x_\ell).
\end{equation}
From~\eqref{eq:prelim}, it follows that  $x_\ell \notin \cH$, otherwise
$x_\ell$ is feasible for~\eqref{mod:P_D} with
$f(x_\ell) < \eqref{mod:P_D}^*$,
a contradiction.
Moreover, \eqref{eq:prelim} implies $\lim_{\ell \to \infty}\tfrac{\eqref{mod:P_D}^*-f(x_\ell)}{h(x_\ell)} = \infty$.

Conversely, assume that $\tfrac{\eqref{mod:P_D}^*-f(x)}{h(x)}$ is unbounded on $\D \setminus \cH$. Thus, given $y\in \R^m$, there is $\hat x \in \D \setminus \cH$ such that $\eqref{mod:P_D}^*>  f(\hat x)+\|y\|_{\infty}h(\hat x) \ge
f(\hat x)+ \sum_{j=1}^m |y_j| h_j(\hat x) \ge f(\hat x)+ \sum_{j=1}^m y_j h_j(\hat x) \ge
\inf_{x \in \D} (f(x)+ \sum_{j=1}^m y_j h_j(x)) = \eqref{mod:LDy}^*$.
\end{proof}

The following examples illustrate the attainment characterization given by Proposition~\ref{prop:dualattain}.

\begin{example}[Lagrangian reformulation]\label{ex:yes}
Let $f(x) = x$, $h(x) = x^2$, and $\D = [0,1]$. Then $\cH = \{0\}$ and $\D \setminus \cH = (0,1]$. Since
\[
\frac{\eqref{mod:P_D}^* - f(x)}{h(x)} = -\frac{1}{x} < 0
\quad \text{for all } x \in (0,1],
\]
this expression is bounded above on $\D \setminus \cH$. Thus, by Proposition~\ref{prop:dualattain}, a Lagrangian reformulation exists. Indeed,
\[
\eqref{mod:P_D}^* = 0 = \min_{x \in [0,1]} \left( x + yx^2 \right), \quad \text{for every } y \ge -1.
\]
\end{example}

\begin{example}[No Lagrangian reformulation]\label{ex:no}
Let $f(x) = x$, $h(x) = x^2$, and $\D = [-1,1]$. Then $\cH = \{0\}$ and $\D \setminus \cH = [-1,0) \cup (0,1]$. In this case,
\[
\frac{\eqref{mod:P_D}^* - f(x)}{h(x)} = -\frac{1}{x}
\]
is not bounded above on $\D \setminus \cH$ (it diverges to $+\infty$ as $x \uparrow 0$ from the left). Hence, by Proposition~\ref{prop:dualattain}, no Lagrangian reformulation exists. We still have $\eqref{mod:P_D}^* = 0 = \eqref{mod:POL}^*$, but the Lagrangian minimization is given by
\[
\min_{x \in [-1,1]} \left( x + yx^2 \right)
= \begin{cases}
y - 1, & \text{if } y < \frac{1}{2}, \\
-\frac{1}{4y}, & \text{if } y \ge \frac{1}{2},
\end{cases}
\]
which is not equal to $0$ for all $y$, confirming the absence of a Lagrangian reformulation.
\end{example}

The next example shows the relevance of the form in which the constraint set is formulated.

\begin{example}[Lagrangian reformulation via alternative $h(x)$]\label{ex:fix}
Continuing from Example~\ref{ex:no}, we consider an equivalent reformulation using \( h(x) = |x| \) instead of \( h(x) = x^2 \). Note that both choices yield the same constraint set:
\[
\cH = \{x \in \D : h(x) = 0\} = \{0\}.
\]
Now,
\[
\frac{\eqref{mod:P_D}^* - f(x)}{h(x)} = -\frac{x}{|x|},
\]
which equals $-1$ for $x > 0$ and $1$ for $x < 0$, so it is bounded above on all of $\D \setminus \cH = [-1,0) \cup (0,1]$. Therefore, by Proposition~\ref{prop:dualattain}, a Lagrangian reformulation \emph{does} exist for this version of the problem.
Indeed,
\[
\eqref{mod:P_D}^* = 0 = \min_{x \in [-1,1]} \left( x + y|x| \right), \quad \text{for all } y \ge 1.
\]

\end{example}

Examples~\ref{ex:yes}-\ref{ex:fix} illustrate the correctness of the Lagrangian reformulation characterization given by \cref{prop:dualattain}. However, this characterization is not practical as it requires the knowledge of~$\eqref{mod:P_D}^*$. For practical purposes, in  \cref{thm:reform} we show that when the set $\D \setminus \cH$ is closed the  problem~\eqref{mod:P_D} has a Lagrangian reformulation. The simplicity of this assumption is pivotal in applying the Lagrangian reformulation to practically relevant instances of nonconvex optimization problems (as discussed in detail in  Section~\ref{sec:apps}). However, as can be seen in \cref{ex:yes}, this condition is not necessary to obtain a
Lagrangian reformulation of~\eqref{mod:P_D}.

\begin{theorem}[Lagrangian reformulation]
\label{thm:reform}
Let $f$, $h_j$ for $j=1, \ldots, m$ be real valued continuous functions and $\D \subseteq \R^n$ be a compact set. Assume $h_j(x) \ge 0$ for all $x \in \D$, $j=1,\dots,m$, $\D\cap \mathcal{H}\neq \emptyset$, and $\D \setminus \cH$ is closed. Then,
\begin{enumerate}[label = (\roman*)]
\item  \eqref{mod:LDy} is a Lagrangian reformulation of~\eqref{mod:P_D} for any $y \in \R^m$ such that $y_j\ge \frac {\eqref{mod:P_D}^* - (D_0)^*}{\tilde{h}}$, $j=1,\dots,m$.
\label{it:dualattain1}

\item For each $y$ such that $y_j>\frac {\eqref{mod:P_D}^* - (D_0)^*}{\tilde{h}}$, $j=1,2,\ldots, m$, we have that~\eqref{mod:LDy} is equivalent to~\eqref{mod:P_D}; that is, the optimal solutions of~\eqref{mod:LDy} coincide with the optimal solutions of~\eqref{mod:P_D}.
\label{it:dualattain2}
\end{enumerate}
\end{theorem}

\begin{proof}
If $\D\setminus \cH = \emptyset$, then, $\D\cap \cH=\D$, and it follows directly that problem~\eqref{mod:P_D} and~\eqref{mod:LDy}, for any $y \in \R^m$ are equivalent; that is, statements~\ref{it:dualattain1} (i.e., $\eqref{mod:LDy}^* = \eqref{mod:P_D}^*$) and~\ref{it:dualattain2} (i.e., $\argmin_x\eqref{mod:LDy} = \argmin_x\eqref{mod:P_D}$) follow. Thus, in what follows, we assume that $\D\cap \mathcal{H}\neq \emptyset$ and closed.

To prove statement~\ref{it:dualattain1}, let $y \in \R^m$ such that $y_j\geq \tfrac{1}{\tilde{h}}(\eqref{mod:P_D}^* - (D_0)^*)$, $j=1,\dots,m$.
let $\epsilon>0$ small enough,
then
$
(D_0)^* \le f_{\epsilon} \le \eqref{mod:P_D}^*$, and $h^\epsilon \ge \tilde h > 0$ by Lemmas~\ref{lem:heps} and~\ref{lem:htilde>0}. Thus
\[
y_j \ge
\frac{f_{\epsilon} - (D_0)^*}{h^\epsilon} \ge 0 \text{ for } j=1,\dots,m.
\]
 By weak duality and Lemma~\ref{lem:elcaso},
 \begin{align}\label{eq:up_bd_pepsilon2}
\eqref{mod:P_D}^* \ge \eqref{mod:POL}^* \ge   \inf_{x \in \D} f(x) + \sum_{j=1}^m y_j h_j(x) \ge f_{\epsilon}.
 \end{align}
 Now from Lemma~\ref{lem:heps}, taking the limit when $\epsilon \downarrow 0$ in~\eqref{eq:up_bd_pepsilon2}, we obtain
$
 \eqref{mod:P_D}^*  =  \inf_{x \in \D} f(x) + \sum_{j=1}^m y_j h_j(x).
 $

To prove statement~\ref{it:dualattain2}, let $x^*\in \D$ be optimal for~\eqref{mod:LDy}. If $x^*\in \D\setminus \mathcal{H}$,
there is $\epsilon > 0$ such that $h(x^*) \ge h^\epsilon  \ge \tilde{h} >0$. Thus,
we have
\begin{align*}
f(x^*)+\sum_{j=1}^m y_j h_j(x^*) > f(x^*) + \eqref{mod:P_D}^* - (D_0)^* \geq  \eqref{mod:P_D}^*
\end{align*}
which contradicts the weak duality (Proposition~\ref{pro:weak}). Therefore, $x^*\in \D\cap \mathcal{H}$, and by statement~\ref{it:dualattain1}, $f(x^*) = \eqref{mod:P_D}^*$, which proves that $x^*$ is optimal for~\eqref{mod:P_D}. On the other hand, if $\tilde{x}^*$ is an optimal solution of~\eqref{mod:P_D}, then $\tilde{x}^*$ is feasible for~\eqref{mod:LDy} with objective value $f(\tilde{x}^*) = \eqref{mod:P_D}^* = \eqref{mod:LDy}^*$, where the last equality follows from statement~\ref{it:dualattain1}. Thus, the optimal solution sets of~\eqref{mod:P_D} and~\eqref{mod:LDy} coincide.
\end{proof}

The next example illustrates the need for the strictly greater than  sign in Theorem~\ref{thm:reform}\ref{it:dualattain2}.

\begin{example}
Let $\D=\{(x_1,x_2)\in \R^2: x_1(1-x_1)=0, \text{ and } 0\leq x_2 \leq 1\}$,  $f(x_1,x_2)=(x_1-1)^2+x_2^2$ and $h(x_1,x_2)= x_1(2-x_1)$. Then, $(D_0)^*=0$, which is attained at $(1,0)$ and $\eqref{mod:P_D}^*=1$, which is attained at $(0,0)$. Also, $h_{\epsilon}=1$ for every sufficiently small $\epsilon>0$, therefore $\tilde{h}=1$. Thus,  problem~\eqref{mod:LDy} for $y=\frac {\eqref{mod:P_D}^* - (D_0)^*}{\tilde{h}} = 1$ is
\begin{align*}
\min_{x\in \D} (x_1-1)^2+x_2^2 + 1(x_1(2-x_1)) = \min_{x\in \D} 1+x_2^2 = 1 = \eqref{mod:P_D}^*
\end{align*}
and the minimum occurs at $\{(0,0),(1,0)\}$. However, the point $(1,0)$ is infeasible for~\eqref{mod:P_D}.
\end{example}

In Section~\ref{sec:apps}, we use Theorem~\ref{thm:reform}
to reformulate distinctive nonconvex optimization problems in a way that is amenable to new computing and algorithmic technologies.

\subsection{Convergence to primal solution  w/o dual attainment}
In this section, we present a primal convergent attainment result which does not require  $\D \setminus \cH$ to be closed. Unsurprisingly, this result is weaker than Theorem~\ref{thm:reform}\ref{it:dualattain2}. Yet, as evidenced in  Section~\ref{sec:MIQPapps}, this result is crucial to extend the application of our reformulation results from the realm of purely integer problems (where the closure condition holds) to the one of mixed-integer problems (where the closure condition does not hold).

\begin{theorem}\label{thm:sol_convergence} Let $f$, $h_j$ for $j=1, \ldots, m$ be real valued continuous functions and $\D \subseteq \R^n$ be a compact set. Assume $h_j(x) \ge 0$ for all $x \in \D$, $j=1,\dots,m$.
Suppose $\D$ is compact and $\D\cap \mathcal{H} \neq \emptyset$. Let $(y_\ell)\subset \R$  be an increasing sequence of positive numbers diverging to infinity, and for each $\ell \in \N$, let
$\hat{x}_\ell \in \argmin \{f(x)+y_\ell \sum_{j=1}^m h_j(x):x\in \D\}$.
Then, every accumulation point of $(\hat x_\ell)$ is an optimal solution of~\eqref{mod:P_D}.
\end{theorem}

\begin{proof}
Let $x^*$ be an accumulation  point of $(\hat{x}_\ell)$, and let $(x_k):=(\hat{x}_{\ell_k})$ be a subsequence of $(\hat{x}_\ell)$
converging to $x^*$. Clearly, $x^*\in \D$, because $\D$ is compact. Let $h(x) =  \sum_jh_j(x)$. Now, we proceed to prove three claims that we will use  to  prove the desired result.

\begin{enumerate}[label={Claim (\roman*)}, wide, labelwidth=0pt, labelindent=0pt]
\item $x^*\in \D\cap \mathcal{H}$ and for any $y \in \R$, $\cL(x^*, y) = f(x^*)$. To prove this claim it is enough to prove that $h(x^*) = 0$. To see this, note that from weak duality (see Proposition~\ref{pro:weak}), for every $k\in \N$, $f(x_k) + y_{\ell_k}h(x_k) \leq \eqref{mod:P_D}^*<\infty$ (because $\D\cap \mathcal{H}\neq \emptyset$). Since
$\D$ is compact, $|f(x_k)|\leq M $ for every $k\in \N$, and some $M>0$.
Therefore, $0\leq h(x_k)\leq \frac{\eqref{mod:P_D}^*-f(x_k)}{y_{\ell_k}} \leq \frac{\eqref{mod:P_D}^*+M}{y_{\ell_k}}$. Thus,
$0 = \lim_{k\to \infty} h(x_k) = h(x^*)$, by continuity of~$h$, which proves that $x^*\in \D\cap \mathcal{H}$. This in turn implies that for any $y \in \R$, $\cL(x^*, y) = f(x^*)$.
\label{claim:1}

\item $(\cL(x_k, y_{\ell_k}))_{k\in \N}$ is a non-decreasing sequence. To prove this claim, let~$k$ and~$l$ be positive integers such that $k<l$. Then
\begin{align*}
\cL(x_k, y_{\ell_k}) &= f(x_k)+y_{\ell_k}h(x_k) \leq f(x_l)+y_{\ell_k}h(x_l),  && \text{(by optimality of $x_k$)}\\
&\leq f(x_l) + y_{\ell_l}h(x_l) = \cL(x_l, y_{\ell_l}), && \text{(since $y_{\ell_k}<y_{\ell_l}$)}
\end{align*}
Notice that the last inequality might not be strict because $h(x_l)$ can be zero.
\label{claim:2}

\item $\cL(x_k,y_{\ell_k}) \xrightarrow{k\to \infty} f(x^*)$. To prove this claim, we show that for any $\epsilon > 0$, $|f(x^*)- \cL(x_k,y_{\ell_k})| < \epsilon$ for a large enough $k \in \N$. For that purpose, let $\varepsilon>0$ and $y>0$ be fixed. By
continuity of $\cL(\cdot,y)$ and the definition of $x^*$, there exists $N\in \N$ such that for
every $l>\max\{N,y\}$,
\begin{equation}
\label{eq:limit}
|\cL(x^*,y) - \cL(x_l,y)|<\varepsilon.
\end{equation}
Now, let $r \in \N$ such that $y_{\ell_r}>l$.
Then, for every $k>r$, we have the following chain of inequalities
\begin{align}
\label{eq:chain}
f(x^*) \geq \cL(x_k,y_{\ell_k}) \geq \cL(x_r,y_{\ell_r}) \geq f(x_r) + y h(x_r) = \cL(x_r,y),
\end{align}
where the first inequality follows from weak duality (see Proposition~\ref{pro:weak}) and
the second follows from~\ref{claim:2}. Thus,
\begin{align*}
0 \leq f(x^*) - \cL(x_k, y_{\ell_k})
& \leq f(x^*) - \cL(x_r, y) && \text{(from~\eqref{eq:chain})}\\ & = |f(x^*) - \cL(x_r,y)|   & \\
& = |\cL(x^*,y) - \cL(x_r,y)| < \varepsilon, && \text{(from~\ref{claim:1} and~\eqref{eq:limit})}
\end{align*}
which completes the claim's proof.
\label{claim:3}
\end{enumerate}

To finish the proof, notice that as the sequence $(y_\ell)$ diverges to infinity, then~\ref{claim:3} together with  Lemma~\ref{lem:lim=sup} imply the statement of the theorem.
\end{proof}

\begin{example}[Convergence to primal solution]\label{ex:convergence_no_reform}
Continuing from Example~\ref{ex:no}, consider the penalized problems
\[
\min_{x \in [-1,1]} \left\{ x + \ell x^2 \right\}, \quad \text{for } \ell= 1,2,\dots
\]
Each objective is strictly convex and differentiable, with minimizer
\(
\hat{x}_\ell = -\tfrac{1}{2\ell}.
\)
Since $y_\ell = \ell \to \infty$, it follows that $\hat{x}_\ell \to 0$. This point lies in $\cH = \{0\}$ and is the unique optimal solution of the constrained problem.

Thus, although a Lagrangian reformulation does not exist in this case (as shown in Example~\ref{ex:no}), the solutions to the penalized problems converge to a feasible and optimal solution, as guaranteed by Theorem~\ref{thm:sol_convergence}.
\end{example}

\section{Distinctive Problems Reformulations}
\label{sec:apps}

In this section, we apply the results obtained in Section 3 to combinatorial and mixed-integer problems, which encompass a very wide class of relevant nonconvex optimization problems.
These reformulations are suitable for new computing and algorithmic technologies.

\subsection{QUBO reformulations of combinatorial problems}
\label{sec:QUBOapps}

Combinatorial optimization problems are a very important class of nonconvex optimization problems as they arise in  applications in many fields~\citep[][]{grotschel1995combinatorial}. As evidenced in Section~\ref{sec:intro}, there is a renewed interest in reformulating combinatorial optimization problems as QUBO problems~\citep[][]{lasserre2016max, lucas2014ising, quintero2022characterization, nannicini2019performance, gusmeroli2022expedis}. In what follows, we show how our results in Section~\ref{sec:reform} contribute to this current and active line of research.

The next corollary specializes the Lagrangian reformulation result in Theorem~\ref{thm:reform} to the case when the feasible set of~\eqref{mod:P_D} is finite and integer, and the equality constraints are given by polynomials with rational coefficients, a case encompassing most  combinatorial optimization problems.

\begin{corollary}
\label{cor:genlasserre} Let $f$ be a real valued continuous function. Let $\D\subset \mathbb{Z}^n$ be bounded and $h_j\in \mathbb{Z}[x]$
be nonnegative for every $x\in \D$  and $j=1, \ldots, m$ and satisfy $\D \cap \cH \neq \emptyset$. Then, problem \eqref{mod:LDy} is a Lagrangian reformulation of~\eqref{mod:P_D} for any $y\in \R^m$ such that $y_j\ge \min\{f(x): x \in \D, x\in\cH\} - \min\{f(x): x \in \D\}$,  $j = 1,2,\ldots, m$.
\end{corollary}
\begin{proof} The set $\D \setminus \cH$ is finite and therefore closed. Also, each $j=1,\dots,m$, $h_j$  takes only integer positive values in $\D\setminus \cH$. Thus for any $\epsilon >0$, it follows that $h^\epsilon \ge 1$. Therefore, $\tilde h \ge 1$. The statement then follows from Theorem~\ref{thm:reform}.
\end{proof}

Corollary~\ref{cor:genlasserre} generalizes the QUBO reformulation result for combinatorial problems of~\citet[][Thm.~2.2, Lem. 2.1]{lasserre2016max}\footnote{\citet[][Thm.~2.2]{lasserre2016max} can be seen as a corollary of~\citet[]{lasserre2016max}[Lem.~2.1] after a natural homogenization argument. It is worth to note that the Ising and quantum devices  do not require the QUBO's objective to be a homogeneous quadratic.}, which applies to the particular case when in problem~\eqref{mod:P_D} the objective $f$ is a quadratic polynomial, the set $\D$ is $\{-1,1\}^n$ (or $\{0,1\}^n$), and the equality constraints $h_j$, $j=1,\dots,m$, are linear with integer coefficients.

The ability to
 incorporate nonlinear equality constraints in Corollary~\ref{cor:genlasserre} is crucial. To demonstrate this, we will examine the set packing (SP) problem~\citep[][]{vemuganti1998applications}. Namely,  let $c \in \R^n$, $Q \in \R^{n \times n}$, and $A \in \{0,1\}^{m \times n}$, and consider the optimization problem
\begin{equation}
\label{eq:linpack}
z_{\SP}:=\min \{x\tr Q x + c\tr x: Ax \le e,  x \in \{0,1\}^n\},
\end{equation}
where $e$ is the vector of all ones in the appropriate dimension. Although the SP problem is formulated as a maximization problem in the literature, here we take (w.l.o.g.) the liberty to formulate it as a minimization problem to match the format of problem~\eqref{mod:P_D}, as well as the results in~\citep{lasserre2016max}.
There is a wide range of crucial combinatorial optimization problems that can be formulated as a set packing problem, including the stable set problem, and problems arising in
railway infrastructure design, ship scheduling, resource constrained project scheduling~\citep{delorme2004grasp}, combinatorial auctions~\citep{ pekevc2003combinatorial}, and forestry~\citep{ ronnqvist2003optimization}.

As detailed next, by adding slack variables, problem~\eqref{eq:linpack} can be written as a linearly constrained binary quadratic optimization problem, and then the results in~\citep{lasserre2016max}  can be used to obtain a QUBO reformulation of the SP problem. On the other hand, we will show that a more efficient QUBO reformulation of the SP problem is obtained by applying Corollary~\ref{cor:genlasserre}, by rewriting the inequality constraint $Ax \le e$ using quadratic equality constraints.

Throughout, assume that~\eqref{eq:linpack} is feasible; that is, that $z_{\SP} < +\infty$. To obtain a QUBO reformulation of~\eqref{eq:linpack}, one can first rewrite~\eqref{eq:linpack} as
\begin{equation}
\label{eq:linpack2}
z_{\SP}  = \min \{x\tr Q x + c\tr x: Ax + s = e,  x \in \{0,1\}^n, s \in \{0,1\}^m\}.
\end{equation}
Then, it follows from~\citep[][Lem.~2.1]{lasserre2016max} that~\eqref{eq:linpack} can be reformulated as the QUBO problem
\begin{equation}
\label{eq:linQUBO}
z_{\SP}  = \min \{x\tr Q x + c\tr x + 2(\rho(c,Q) + 1)\|Ax + s - e\|^2: x \in \{0,1\}^n, s \in \{0,1\}^m\},
\end{equation}
where $\rho(c,Q) := \max\{|r^1_{c,Q}|, |r^2_{c,Q}|\}$, $r^1_{c,Q}:= \min\{ \langle Q X \rangle + c\tr x: X \succeq 0, X_{ii}= x_i, i=1,\dots,n\}$, $r^2_{c,Q}:= \max\{ \langle Q X \rangle + c\tr x: X \succeq 0, X_{ii}= x_i, i=1,\dots,n\}$~\citep[see,][eq.~(2.2)]{lasserre2016max}. Above, given symmetric matrices $X, Y\in \R^n\times\R^n$, $Y \succeq 0$ indicates that the matrix is positive semidefinite, and $\langle X, Y \rangle$ indicates the trace of the product of $X$ and $Y$.

While the reformulation results in~\citep{lasserre2016max} apply to linearly constrained binary quadratic optimization problems only, the reformulation results presented here apply also to nonlinearly constrained binary quadratic optimization problems. As a result, one can obtain a different QUBO reformulation of~\eqref{eq:linpack}. First, note that~\eqref{eq:linpack} can be reformulated as the nonlinearly constrained binary quadratic optimization problem
\begin{equation}
\label{eq:nonlinpack}
z_{\SP}=\min \{x\tr Q x + c\tr x: A_{ki}A_{kj}x_ix_j = 0,\, 1 \le k \le m,\, 1 \le i < j \le n,\,  x \in \{0,1\}^n\}.
\end{equation}
Taking $f(x) = x\tr Q x + c\tr x$, $\D = \{0,1\}^n$ and $h_{(ijk)} = A_{ki}A_{kj}x_ix_j$ for $1 \le k \le m$, $1 \le i < j \le n$, it is clear that all the assumptions of Corollary~\ref{cor:genlasserre} are satisfied. Further, note that $r^2_{c,Q} \ge z_{\SP}$, and
$r^1_{c,Q} \le \min \{x\tr Q x + c\tr x:  x \in \{0,1\}^n\}$. Thus, it follows from Corollary~\ref{cor:genlasserre} that~\eqref{eq:linpack} can be reformulated as the QUBO problem
\begin{equation}
\label{eq:nonlinQUBO}
z_{\SP}  = \min \left \{x\tr Q x + c\tr x + \frac{1}{2}(r^2_{c,Q} - r^1_{c,Q} + 1) \sum_{k=1}^m \sum_{1 \le i < j \le n} A_{ki}A_{kj}x_ix_j: x \in \{0,1\}^n \right \}.
\end{equation}

The number of binary decision variables in~\eqref{eq:linQUBO} is $n+m$, while the number of decision variables in~\eqref{eq:nonlinQUBO} is $n$. Given the limitations of current and near-term quantum device technologies, this is a significant difference between these two reformulations. The number of {\em qubits} on quantum devices designed to solve QUBO problems is limited. Moreover, the connectivity of the qubits in these devices is also very limited. Indeed, one of the most powerful quantum devices has a ``15-way qubit connectivity''~\citep{DWave}. Loosely speaking, this means that qubits are arranged on sparsely connected cliques of 15 qubits, which means that to {\em embed} a QUBO problem with $N$ variables, a much larger number of qubits $N' >> N$ is needed~\citep[][]{zbinden2020embedding}. There is a rich amount of literature on deriving algorithms to find the optimal  embeddings (i.e., using the least number of qubits), and for specific connectivity patterns, the optimal number of qubits $N'$ is quadratic on $N$~\citep{date2019efficiently}.

\subsection{Reformulations of Mixed-Integer problems}
\label{sec:MIQPapps}

We now consider the more general setting of mixed-integer optimization,
where the goal is to minimize a function over a feasible set defined by linear constraints
with continuous and binary variables. Specifically, let $A\in \R^{m\times n}, b\in \R^m$, $J\subseteq \{1,\dots,n\}$, $f$ be a real-valued continuous function, and consider the problem
\begin{equation}
\tag{B}
\label{mod:OP0}
\begin{array}{lllll}
\displaystyle \min & f(x) \\
\st & Ax=b\\
     & x \in [0,1]^n \cap \B_J,
\end{array}
\end{equation}
where $\B_J:= \{x \in \R^n: x_i \in \{0,1\}, \text{for all } i \in J\}$.

We consider different Lagrangian relaxations of problem~\eqref{mod:OP0}. Given $y \in \R$ one can relax only the linear constraints of~\eqref{mod:OP0} to obtain
\begin{equation}
\label{mod:OP1}
\tag{B$^1_y$}
\min \{ f(x) + y \|Ax-b\|^2: x\in [0,1]^n\cap \B_J\}.
\end{equation}
Problem~\eqref{mod:OP1} is suitable for Ising machines, and quantum devices using quantum annealing or QAOA algorithms when all variables are binary; effective methods for box-constrained mixed-integer or continuous polynomial optimization when the objective is a polynomial~\citep{buchheim2014box, letourneau2023efficient}; and branch-and-bound techniques combined with methods for box-constrained quadratic optimization when the objective is quadratic~\citep[][]{kim2010tackling, bonami2018globally}.

Alternatively, given $z \in \R$, one can relax only the binary variable constraints of~\eqref{mod:OP0} to obtain
\begin{equation}
\label{mod:OP2}
\tag{B$^2_z$} \min \{ f(x) + z \textstyle\sum_{i \in J} x_i(1-x_i) : Ax = b, x\in [0,1]^n\},
\end{equation}
or one can relax both the linear and binary variable constraints of~\eqref{mod:OP0} to obtain
\begin{equation}
\label{mod:OP3}
\tag{B$^3_{yz}$}
\min \{ f(x) + y \|Ax-b\|^2 + z \textstyle\sum_{i \in J} x_i(1-x_i) : x\in [0,1]^n\}.
\end{equation}
Problems of the form~\eqref{mod:OP2} or~\eqref{mod:OP3} are amenable to a wide range of global nonlinear optimization solvers~\citep[][for a recent benchmark of such solvers]{kronqvist2019review}. Since \eqref{mod:OP2} or~\eqref{mod:OP3} are continuous optimization problems, there is no need to use branch-and-bound techniques for binary variables that are implemented in general mixed-integer nonlinear optimization solvers. In fact, recently positive results for solving continuous relaxations of QUBO problems
were obtained in~\citep{bartmeyer2024benefits}. Additionally, problems of the form~\eqref{mod:OP3} are amenable to the use of
recent quantum optimization algorithms like quantum gradient descent~\citep{rebentrost2019quantum}, quantum Hamiltonian descent (QHD)~\citep{2023QHD}, and quantum Langevin dynamics~\citep{chen2023quantum}, designed for quantum devices.

With this in mind, we establish several equivalences between these different Lagrangian relaxations and problem~\eqref{mod:OP0}, providing a complementary tool, alongside novel computational and algorithmic paradigms, to solve~\eqref{mod:OP0}.

We first present equivalence results for the case in which problem~\eqref{mod:OP0} is a purely integer problem (i.e., when $J= \{1,\dots,n\}$). Before presenting these equivalences in Theorem~\ref{thm:equivBPpure}, consider the following example, which illustrates the approach  used  to derive them.

\begin{example}[Equivalence between~\eqref{mod:OP0} and~\eqref{mod:OP2}]
\label{ex:equiv1a}
Consider the optimization problem of the form~\eqref{mod:OP0} explicitly defined by $\min\{f(x) := x(x-\tfrac{1}{2})(x-1): x \in [0,1] \cap \{0,1\}\}$. Clearly, $f$ is neither concave nor convex on $[0,1]$, \eqref{mod:OP0}$^* = 0$, and the set of optimal solutions of~\eqref{mod:OP0} is $\{0,1\}$. Now construct the Lagrangian relaxation of the form~\eqref{mod:OP2} (i.e., obtained by relaxing the binary variable constraint) explicitly defined by $\min\{\cL(x,z):=x(x-\tfrac{1}{2})(x-1) + zx(1-x): x \in [0,1]\}$. For $z > \tfrac{3}{2}$, $\cL(x,z)$ is strictly concave on~$[0,1]$, and thus the optimal solution of~\eqref{mod:OP2} is attained at the extreme points of the the interval $[0,1]$, and $\cL(0,z) =\cL(1,z) = 0$. Thus, for  $z > \tfrac{3}{2}$,~\eqref{mod:OP2} is equivalent to~\eqref{mod:OP0}.
\end{example}

In Example~\ref{ex:equiv1a}, we exploit the property that after adding a large enough multiple of an appropriate quadratic to the original function~$f$, the resulting function~$\cL(x,z)$ becomes strictly concave. Thus, at that point,
the minimum of~$\cL(x,z)$ over the interval is attained at one of the extreme points of the interval and
the desired Lagrangian reformulation is obtained. As the next lemma formally states, $L$-smooth functions have the aforementioned property in the more general case in which the feasible set is defined by a polytope.

\begin{lemma}
\label{lem:lip}
Let $P \subseteq [0,1]^n$ be a polytope, and $f$ be $L$-smooth on $P$.
Then, for every $z > \tfrac{L}{2}$, $\argmin\{f(x) + z \sum_{i=1}^n x_i(1-x_i): x \in P\} \cap E \neq \emptyset$, where $E$ is the set of extreme points of P.
\end{lemma}

\begin{proof}
The result follows from the fact that a $L$-smooth function $f$ satisfies that $\tfrac{L}{2}\|\cdot\|_2 -f$ is convex~\citep[][Def. 6]{giselsson2016linear}.
\end{proof}

With Lemma~\ref{lem:lip} at hand, we now state the following equivalences.

\begin{theorem}\label{thm:equivBPpure} Let $A\in \R^{m\times n}, b\in \R^m$, $J= \{1,\dots,n\}$, and $f$ be  $L$-smooth on $[0,1]^n$. Assume that~\eqref{mod:OP0} is feasible. Then, there exist $y'\in \R$ and $z'\in \R$ such that for all $y > y'$ and all $z > z'$, problems \eqref{mod:OP0}, \eqref{mod:OP1}, \eqref{mod:OP2}, and \eqref{mod:OP3} are all equivalent to each other.
That is,  the optimal solutions of all these problems coincide.
\end{theorem}

\begin{proof}
It suffices to show that the equivalence statement follows between~\eqref{mod:OP0} and~\eqref{mod:OP1},~\eqref{mod:OP0} and~\eqref{mod:OP2}, and~\eqref{mod:OP3} and~\eqref{mod:OP1}, as the other pairwise problem equivalences then follow by transitivity.

To prove
that~\eqref{mod:OP0} is equivalent to~\eqref{mod:OP1},
notice that in this case, problem~\eqref{mod:OP0} is equivalent to $\min\{f(x): \|Ax-b\|^2 = 0, x \in \{0,1\}^n\}$.
Since all the conditions of Theorem~\ref{thm:reform} are satisfied, it follows that
there exists $y'\in \R$ such that for every $y> y'$,~\eqref{mod:OP0} is equivalent to \eqref{mod:OP1}.

To prove
that \eqref{mod:OP0} is equivalent to~\eqref{mod:OP2},
note that~\eqref{mod:OP0} is equivalent to $\min\{f(x): x_i(1-x_i) = 0, i=1,\dots,n: x \in \{x \in [0,1]^n: Ax = b\}\}$. Then, weak duality (Proposition~\ref{pro:weak}) implies that
\begin{equation}
\label{eq:weak02}
\eqref{mod:OP0}^*\geq \eqref{mod:OP2}^*.
\end{equation}
We will argue that there exists $z'$ such that for every $z> z'$, $\eqref{mod:OP0}^*\leq \eqref{mod:OP2}^*$ and thus $\eqref{mod:OP0}^*= \eqref{mod:OP2}^*$.
 To see this, note that from Lemma~\ref{lem:lip}, it follows that for any $z > \frac{L}{2}$, problem~\eqref{mod:OP2} is equivalent to
\begin{equation}
\label{mod:R}
\min \left \{ f(x) + z \sum_{i=1}^n x_i(1-x_i) :  x\in E \right \},
\end{equation}
where $E$ is the set of extreme points of $\{x\in [0,1]^n: Ax = b\}$.  Furthermore,
by Theorem~\ref{thm:reform}, there is $\tilde{z} \in \R$ such that for every $z>\tilde{z}$, problem~\eqref{mod:R} is equivalent to
\begin{equation}
\label{mod:Rp}
\min\{ f(x): x_i(1-x_i) = 0, i=1,\dots,n: x\in E\}.
\end{equation}
Let $z'=\max\{\tfrac{L}{2}, \tilde{z}\}$, then, for every $z>z'$, we have that~\eqref{mod:OP2} is equivalent to~\eqref{mod:R}, and that~\eqref{mod:R} is equivalent to~\eqref{mod:Rp}.
Thus,~\eqref{mod:OP2} is equivalent to~\eqref{mod:Rp}, implying that: $\eqref{mod:OP2}^* = \eqref{mod:Rp}^* \ge \eqref{mod:OP0}^*$, which together with~\eqref{eq:weak02} implies that $\eqref{mod:OP2}^* = \eqref{mod:OP0}^*$. Now that we have proven that the optimal values coincide, we need to prove the the optimal solution coincide as well. Assume
that $x^*$ is an optimal solution of~\eqref{mod:OP2}, then from the equivalence between~\eqref{mod:OP2} and~\eqref{mod:Rp}, it follows that $x^*$ is a feasible solution of~\eqref{mod:OP0} with objective $f(x^*) = \eqref{mod:OP0}^*$. Thus, $x^*$ is an optimal solution of~\eqref{mod:OP0}.
On the other hand, if $\tilde{x}^*$ is an optimal solution of~\eqref{mod:OP0}, then $\tilde{x}^*$ is feasible for~\eqref{mod:OP2} with objective value $f(\tilde{x}^*) = \eqref{mod:OP0}^* = \eqref{mod:OP2}^*$. Thus, $\tilde{x}^*$ is an optimal solution of~\eqref{mod:OP2}.

To prove
that that~\eqref{mod:OP1} is equivalent to~\eqref{mod:OP3},
fix $y \in \R$ and note that  from Lemma~\ref{lem:lip}, it follows that for any $z > \tfrac{L}{2} + y \|A\|^2$,  the optimal solutions of~$\eqref{mod:OP3}$ are all in $\{0,1\}^n$.
\end{proof}

The following example illustrates how one of the equivalences stated in Theorem~\ref{thm:equivBPpure} is obtained.

\begin{example}
Consider the optimization problem of the form~\eqref{mod:OP0} explicitly defined by $\min\{x(x+1): x \in [\tfrac{1}{2}, 1] \cap \{0,1\}\}$.
It is easy to see
that $\eqref{mod:OP0}^* = 2$ and the optimal solution of~\eqref{mod:OP0} is $x^*=1$.
Now, construct the Lagrangian relaxation of the form~\eqref{mod:OP2} (i.e., obtained by relaxing the binary variable constraint) explicitly defined by $\min\{ \cL(x,z):=x(x+1) + zx(1-x):  x \in [\frac{1}{2},1]\}$. It is easy to analyze problem~\eqref{mod:OP2}, to reach the following conclusions:
\begin{enumerate}
\item  For $0<z<5$	the optimal solution of~\eqref{mod:OP2} is~$x(z)^* = \tfrac{1}{2}$ with objective value $\eqref{mod:OP2}^* = \tfrac{z+3}{4}$. Notice that for $1\le z<5 $, the objective function of~\eqref{mod:OP2} is concave and thus attains its minimum at an extreme point the interval $[\tfrac{1}{2}, 1]$.
\item For $z \ge 5$, the optimal solution of~\eqref{mod:OP2} is~$x(z)^* = 1$ with objective value $\eqref{mod:OP2}^* = 2$; that is, for $z\ge 5$, problems~\eqref{mod:OP2} and~\eqref{mod:OP0} are equivalent.
\label{it:exlp2}
\end{enumerate}
Thus, unlike in Example~\ref{ex:equiv1a}, where it was enough to set $z$ to a value at which the optimal solution of~\eqref{mod:OP2} would be attained at the extreme point of the underlying polyhedron defining its feasible set, in this example, a larger value of $z$ is needed.
\end{example}

The following example illustrates the need for the $L$-smoothness condition in Theorem~\ref{thm:equivBPpure}.

\begin{example}[No equivalence without $L$-smoothness]
\label{ex:noconvergence}
Consider the function $f(x) = -\sqrt{x} + 2x (4x^2 - 2x -1)$  defined over the interval $[0,1]$ and the optimization problem of the form~\eqref{mod:OP0} explicitly defined by $\min\{f(x): x \in [0,1] \cap \{0,1\}\}$. The function $f$ is not $L$-smooth, \eqref{mod:OP0}$^* = 0$, and $0$ is the unique optimal solutions of~\eqref{mod:OP0}.
Now, construct the Lagrangian relaxation of the form~\eqref{mod:OP2} (i.e., obtained by relaxing the binary variable constraint) explicitly defined by $\min\{ \cL(x,z):=f(x) + zx(1-x):  x \in [0,1]\}$. Figure~\ref{fig:function} shows the plot of $\cL(x,z)$ for different values of $z$. Figure~\ref{fig:function} (right) shows $x(z)^*$ the solution set to the equation $\frac{\partial \cL(x,z)}{\partial x} = 0$.
For  $0 < z \le 5$, problem~\eqref{mod:OP2} attain its minimum at $\tfrac{1}{2} < x_3^*(z) < 1 $. Further,  when $z>5$, problem~\eqref{mod:OP2} attains its (global) minimum at $0 < x_1^*(z) < \tfrac{1}{2} $. It also has a local maximum $0< x_2^*(z) < \tfrac{1}{2}$ and a local min at $\tfrac 12 < x_3^*(z) < 1$, as shown in Figure~\ref{fig:function}. There is no $z \in \R$ for which~\eqref{mod:OP2} attains its minimum at the extreme points of the interval $[0,1]$, which illustrates the need for the smoothness condition in Lemma~\ref{lem:lip}. In turn, there is no $z \in \R$ for which~\eqref{mod:OP2} is equivalent to~\eqref{mod:OP0}, which illustrates the need for the smoothness condition in Theorem~\ref{thm:equivBPpure}.

However, note that $\lim_{z \to \infty} \{x_1^*(z)\} = \{0\}$ (see Figure~\ref{fig:function} (right)), and $\lim_{z \to \infty} \eqref{mod:OP2}^* = 0$. So loosely speaking, in the limit there is equivalence between~\eqref{mod:OP0} and~\eqref{mod:OP2}.
\end{example}

\begin{figure}[H]
  \begin{subfigure}[b]{0.525\linewidth}
    \centering
    \includegraphics[width=0.8\linewidth]{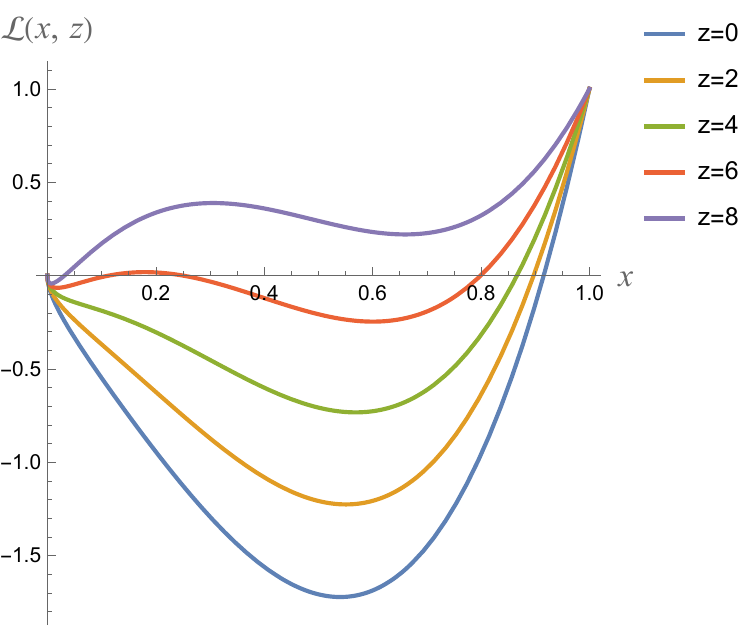}
  \end{subfigure}
  \begin{subfigure}[b]{0.475\linewidth}
    \centering
    \includegraphics[width=0.75\linewidth]{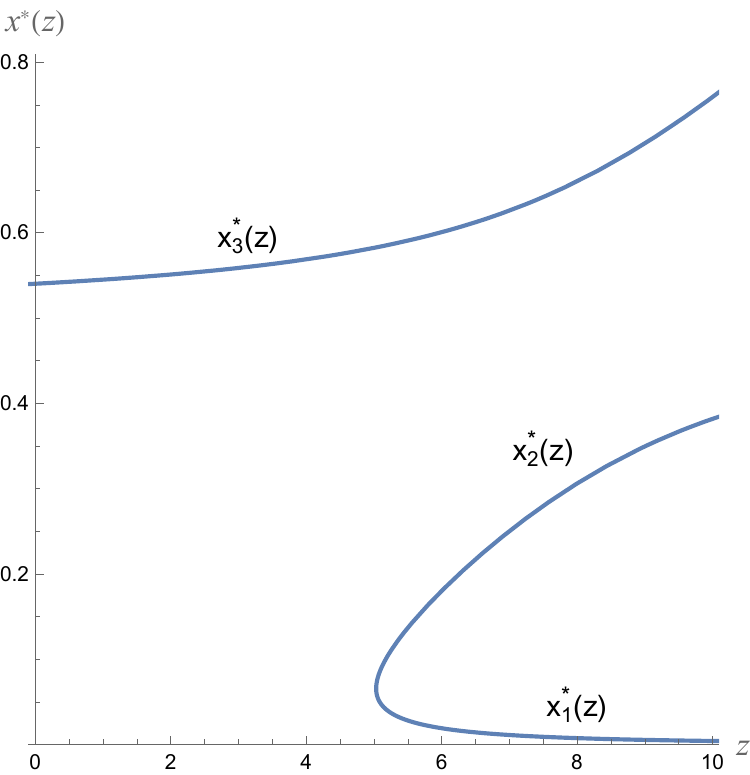}
  \end{subfigure}
  \caption{Left: Plots of functions $\cL(x,z)$ defined in Example~\ref{ex:noconvergence} in the interval $[0.1]$. Right: Plot of solutions to $\frac{\partial \cL(x,z)}{\partial x} = 0$ for different values of $z$. \label{fig:function}
}
 \end{figure}

The limit behavior illustrated in Example~\ref{ex:noconvergence} is formally shown to hold in the next theorem for problems of the form~\eqref{mod:OP0} with an objective that is not $L$-smooth or that are mixed-integer (i.e., when $J \subset \{1,\dots,n\}$). For that purpose,  let us use $\nint(x)$ to denote the unique integer vector $z\in \Z^n$ whose component $z_i$ corresponds to the nearest integer to $x_i$, for $i=1,\ldots, n$ (break ties arbitrarily, for instance by rounding to the closest even integer). Also, for $J\subseteq \{1,2\ldots, n\}$, let $x_J \in \R^J$ denote the vector $x \in \R^n$ restricted to the components indexed by $J$.

\begin{theorem}\label{thm:equivBPmix} Let $A\in \R^{m\times n}, b\in \R^m$, $J\subseteq \{1,\dots,n\}$, and $f$ be a real-valued continuous function. Further, assume that \eqref{mod:OP0} is feasible. Then $\eqref{mod:OP0}^* = \lim_{y \to \infty}\eqref{mod:OP1}^* = \lim_{z \to \infty}\eqref{mod:OP2}^* = \lim_{y,z \to \infty}\eqref{mod:OP3}^*$.
Furthermore, there exist a positive number $z'\in \mathbb{R}$ such that for every $z\geq z'$,
\begin{enumerate}[label = (\roman*)]
\item If $x$ is an optimal solution of~\eqref{mod:OP2}, there exists an optimal solution $x^*$ for \eqref{mod:OP0} such that $x^*_J=\nint(x_J)$.
\label{it:2a}
\item For each $y>0$, if $x$ is an optimal solution for~\eqref{mod:OP3}, there exists an optimal solution $x^*$ for \eqref{mod:OP1} such that $x^*_J=\nint(x_J)$.
\label{it:2b}
\end{enumerate}
\end{theorem}

\begin{proof} The equalities $\eqref{mod:OP0}^* = \lim_{y \to \infty}\eqref{mod:OP1}^* = \lim_{z \to \infty}\eqref{mod:OP2}^* = \lim_{y,z \to \infty}\eqref{mod:OP3}^*$ follow directly from Lemma~\ref{lem:lim=sup}. We now proceed to prove statement~\ref{it:2a}. The proof of statement~\ref{it:2b} is analogous. For each $z > 0$, let
$x^z \in \argmin \{ f(x) + z \textstyle\sum_{i \in J} x_i(1-x_i) : Ax = b, x\in [0,1]^n\}$.
We claim that for every $\epsilon>0$, there is $z'$ such that for each $z>z'$, there exists $\tilde{x}^z$, an optimal solution of~\eqref{mod:OP0}, such that
$\|x^z - \tilde{x}^z\|<\epsilon$.  By using the claim with $\epsilon = 1/2$ the statement of the theorem follows, as $\|x^z-\tilde{x}\|< \frac{1}{2}$ implies $\nint(x^{z}_J) = \tilde{x}_J$.

To finish the proof, we only need to prove the  claim. Let $S$  be the set of optimal solutions of~\eqref{mod:OP0}.
Assume, by sake of contradiction, that there is an $\epsilon>0$ such that for all $z'>0$, there exists $z>z'$ such that $\dist(x^z,S)>\epsilon$. In particular, for $\ell = 1,2,\dots$ there is $z_\ell > \ell$  such that $\dist(x^{z_\ell},S)>\epsilon$.
Since $(x^{z_\ell}) \in  \{x\in [0,1]^n:Ax = b\}$ which is compact, the sequence  $(x^{z_{\ell}})_{\ell > 0}$ has an accumulation point $\tilde{x}^*$.
By Theorem~\ref{thm:sol_convergence}, $\tilde{x}^*$ is an optimal solution of~\eqref{mod:OP0} contradicting $\dist(x^{z_{\ell}},S)>\epsilon$.
\end{proof}

\section{Final Remarks}

The results provided in the article characterize the equivalence between equality-constrained nonconvex optimization problems and their Lagrangian relaxation. These results bridge a gap in the literature between general
Lagrangian duality results for nonconvex optimization~\citep[][among many others]{rockafellar1974augmented, bertsekas1976penalty, huang2003unified, dolgopolik2016unifying, estrin2020implementing} that often lack the constructive dual attainment characterizations that are necessary to obtain Lagrangian reformulations of the original problem, and Lagrangian reformulation results for particular classes of optimization problems~\citep[][]{lasserre2016max, lucas2014ising, bartmeyer2024benefits, date2021qubo, quintero2022characterization, gusmeroli2022expedis, feizollahi2017exact, bhardwaj2022exact, gu2020exact}. We consider a general setting in which the objective is a continuous function, and Lagrangian reformulations are obtained by relaxing general (possibly non-linear) equality constraints. Unlike other related articles, we characterize the equivalence between a problem and its Lagrangian reformulation in terms of their optimal solutions, instead of their objective value alone. This is particularly relevant for mixed-integer problems. Namely,
as stated in Remark~\ref{rem:round} below, Theorem~\ref{thm:equivBPmix} provides a way to solve mixed-integer problems of the form~\eqref{mod:OP0} using continuous optimization solvers, even when the desired equivalences with the problem's Lagrangian relaxations might only occur in the limit. Basically, one can obtain the optimal value of the problem's integer variables by rounding the value obtained for these variables by solving the Lagrangian relaxations~\eqref{mod:OP2} and~\eqref{mod:OP3}, and then solve a reduced problem to obtain the optimal value of the continuous variables.

\begin{remark}[Rounding to mixed-integer solution]
\label{rem:round} Let $x^*$ denote an optimal solution of problem~\eqref{mod:OP0}. Note that by Theorem~\ref{thm:equivBPmix}\ref{it:2a}, if after solving~\eqref{mod:OP2} for some fixed $z \in \R$, its optimal solution $x(z)^*$ satisfies $\max\{|x(z)^*_i - x^*_i|: i \in J\} < \tfrac{1}{2}$, then problem~\eqref{mod:OP0} can be solved by solving the continuous optimization problem
$\min \{ \tilde{f}(x_{J^c}): A_{J^c}x_{J^c}=(b-A_J\nint(x(z)^*_J)), x_i \in [0,1], i \in J^c\}$, where $J^c$ is the complement of $J$, $\tilde{f}(x_{J^c})$ is the function on $x_{J^c}$ obtained after fixing the values of $x_J$ to $x^*_J$, and $A_J$ are the columns of $A$ indexed by the index set $J$. An analogous statement applies when solving~\eqref{mod:OP3} using Theorem~\ref{thm:equivBPmix}\ref{it:2b}.
\end{remark}

Remark~\ref{rem:round} motivates giving a formal look at numerically solving mixed-integer nonconvex optimization problems using continuous nonconvex (or convex if the problem's nonconvexity stems only from integrality constraints) optimization solution approaches  by relaxing the integrality constraints; an approach that is heuristically used in many practical applications (consider, e.g., the explicit penalization methods used in topology optimization~\citep{bendsoe2013topology}).

We finish by noting that our results leave open an interesting question. It is not difficult to construct an example of a mixed-integer problem of the form~\eqref{mod:OP0} with a $L$-smooth objective for which its
associated Lagrangian relaxation~\eqref{mod:OP2} finitely converges to the optimal solution and objective of~\eqref{mod:OP0}. This leaves open the possibility that Theorem~\ref{thm:equivBPpure} might hold even if problem~\eqref{mod:OP0} is not a pure binary problem. For the particular case in which $f$ is a convex quadratic function it follows from the results in~\citep{feizollahi2017exact, bhardwaj2022exact, gu2020exact} that the optimal value of~\eqref{mod:OP1} finitely converges to the optimal value of~\eqref{mod:OP0} (i.e., there is a $y$ for which $\eqref{mod:OP1}^* = \eqref{mod:OP0}^*$). However, it is not clear whether this holds when considering more general objective functions and Lagrangian relaxations of the form~\eqref{mod:OP2} and~\eqref{mod:OP3}.

\section*{Acknowledgements}

The first and third author of this article were supported by  funding from the Defense Advanced Research Projects
Agency (DARPA), ONISQ grant W911NF2010022, titled The Quantum Computing Revolution
and Optimization: Challenges and Opportunities. Also, this research used resources of the Oak Ridge Leadership Computing Facility, which is a DOE Office of Science User Facility supported under Contract DE-AC05-00OR22725.

\bibliography{Lagrangian_BibTeX}

\end{document}